\documentclass[10pt, reqno]{amsart}
\usepackage{amsmath, amsthm, amscd, amsfonts, amssymb, graphicx, color}
\usepackage[bookmarksnumbered, colorlinks, plainpages]{hyperref}
\hypersetup{colorlinks=true,linkcolor=red, anchorcolor=green, citecolor=cyan, urlcolor=red, filecolor=magenta, pdftoolbar=true}
\usepackage{mathrsfs}

\newtheorem{theorem}{Theorem}[section]
\newtheorem{lemma}[theorem]{Lemma}

\newtheorem{corollary}[theorem]{Corollary}
\theoremstyle{definition}

\newtheorem{example}[theorem]{Example}

\theoremstyle{remark}
\newtheorem{remark}[theorem]{Remark}
\numberwithin{equation}{section}

\begin{document}
\setcounter{page}{1}

\title[Drinfeld module]{Drinfeld modules as noncommutative tori}

\author[Nikolaev]
{Igor V. Nikolaev$^1$}

\address{$^{1}$ Department of Mathematics and Computer Science, St.~John's University, 8000 Utopia Parkway,  
New York,  NY 11439, United States.}
\email{\textcolor[rgb]{0.00,0.00,0.84}{igor.v.nikolaev@gmail.com}}





\begin{abstract}
The Drinfeld module is a  tool of the explicit class field theory for the function fields. 
We first observe a similarity of such modules with the noncommutative tori,
and then use it to develop an explicit class field theory for the number fields.
The case of the imaginary quadratic number fields is treated in detail. 
\end{abstract}

\maketitle

\section{Introduction}
The noncommutative torus is a universal $C^*$-algebra $\mathscr{A}_{\theta}$ 
generated by a pair of the unitary operators $U$ and $V$ acting on a Hilbert space $\mathcal{H}$
and satisfying the commutation relation:
\begin{equation}\label{eq1.1}
VU=e^{2\pi i \theta}UV
\end{equation}
for a real constant $\theta\in \mathbf{R}/\mathbf{Z}$.  Such tori
are critical in noncommutative geometry and its applications  \cite[Section 1.1]{N}. 
We say that the noncommutative torus is rational or algebraic if the constant
is a rational or an irrational algebraic number $\alpha\in\overline{\mathbf{Q}}$, respectively; hence the
notation  $\mathscr{A}_{\mathbf{Q}}$ and  $\mathscr{A}_{\alpha}$. 
The Mundici algebra [Mundici 1988,  Section 3]  \cite{Mun1} is a cluster $C^*$-algebra \cite{Nik1}
$\mathbb{A}(S_{1,1})$ corresponding to the ideal triangulation of the once punctured
torus $S_{1,1}$; we refer the reader to [Williams 2014] \cite{Wil1} for the notation and details.
It is known that  $\mathscr{A}_{\theta}=\mathbb{A}(S_{1,1})/I_{\theta}$, where $I_{\theta}$ 
is a two-sided primitive ideal of $\mathbb{A}(S_{1,1})$  [Mundici 1988,  Theorem 3.1]  \cite{Mun1}.

\medskip
Let $A=\mathbf{F}_q[T]$ be the ring of polynomials  in one variable over a finite field $\mathbf{F}_q$
[Rosen 2002] \cite[Chapter 1]{R}. Denote by  $k=\mathbf{F}_q(T)$  the field of rational functions over 
$\mathbf{F}_q$.  Recall that the abelian extensions of $k$ are constructed using the  Drinfeld modules of rank 1 
[Rosen 2002] \cite[Chapters 12 \& 13]{R}. Roughly speaking, such modules mimic  exponential functions 
in the case of function fields. Namely, a polynomial $f\in k[x]$ is called additive
in the ring $k[x,y]$ if $f(x+y)=f(x)+f(y)$.  If $char ~k=p$, then 
 the polynomial $\tau_p(x)=x^p$ is additive and each
additive polynomial has the form $a_0x+a_1x^p+\dots+a_rx^{p^r}$. 
The set of all additive polynomials is closed under addition and composition 
operations thus giving us a ring of  the non-commutative polynomials $k\langle\tau_p\rangle$
defined by  the commutation relation:
\begin{equation}\label{eq1.2}
\tau_p a=a^p\tau_p
\end{equation}
for all  $a\in k$.  The  Drinfeld module $Drin_A(k)$  is a homomorphism
$\rho: A\to k\langle\tau_p\rangle$,
such that for all $a\in A$ the constant term of $\rho_a$ is $a$ and 
$\rho_a\not\in k$ for at least one $a\in A$.  
The  module $Drin_A(k)$ is called trivial, if $\rho_a\in k$ for all $a\in A$.  
The class field theory  says that for each non-zero $a\in A$ the function 
field $k\left(\Lambda_{\rho}[a]\right)$  is an abelian extension of $k$ 
whose Galois group is isomorphic to a subgroup of $\left(A/(a)\right)^*$,
where   $\Lambda_{\rho}[a]=\{\lambda\in\bar k ~|~\rho_a(\lambda)=0\}$
is a torsion submodule of the non-trivial  Drinfeld module of rank 1  [Rosen 2002] \cite[Proposition 12.5]{R}.

The reader can observe a similarity of the commutation relations (\ref{eq1.1}) and (\ref{eq1.2})
coming from an analog of the Fermat's Little Theorem for the function fields. 
Namely, let $\mathcal{P}\in A$ be an irreducible polynomial and $a\in A$ be a polynomial 
not divisible by $\mathcal{P}$. Then $a^{|\mathcal{P}|-1}\equiv 1\mod\mathcal{P}$, where 
$|\mathcal{P}|=q^{\deg\mathcal{P}}$  [Rosen 2002] \cite[Proposition 1.8 \& Corollary]{R}. 
One can redefine $\tau_p(\mathcal{P}):=x^{q^{\deg\mathcal{P}}}$ [Rosen 2002] \cite[p.~199]{R}
so that the commutation  relation (\ref{eq1.2}) takes the form: 
\begin{equation}\label{eq1.3}
\tau_p(\mathcal{P})a=a^{q^{\deg\mathcal{P}}}\tau_p(\mathcal{P}).
\end{equation}
After reduction of (\ref{eq1.3}) modulo $\mathcal{P}$ and  the Fermat's Little Theorem, 
one gets the commutation relation:
\begin{equation}\label{eq1.4}
\tau_p(\mathcal{P}) a=(1 \mod\mathcal{P}) ~a\tau_p(\mathcal{P}),
\end{equation}
where  $a\in A$ is not divisible by $\mathcal{P}$.  It is immediate, 
that relations  (\ref{eq1.1}) and (\ref{eq1.4}) 
are equivalent by the substitution $U=a, ~V=\tau_{\mathcal{P}}$ 
and $e^{2\pi i\theta}=1 \mod\mathcal{P}$. 
Notice that in the last expression the inverse ${1\over 2\pi i}\log ~(1 \mod\mathcal{P})$ is 
a  discrete logarithm equal to an algebraic number $\alpha$, see item (iii) of Theorem \ref{thm1.1}.

The aim of our note is a precise relation between the rings $k\langle\tau_p\rangle$,  noncommutative tori $\mathscr{A}_{\theta}$ and 
 the Mundici algebra $\mathbb{A}(S_{1,1})$; see  Theorem \ref{thm1.1}. 
In particular,  it is proved  that the Drinfeld module $Drin_A(k)$ is equivalent  to an inclusion
 $Spec~(A) \subset\overline{\mathbf{Q}}/\mathbf{Z}$,
where $Spec~(A)$ is the spectrum of the ring  $A$ and $\overline{\mathbf{Q}}/\mathbf{Z}$ are algebraic numbers of
the unit interval $\mathbf{R}/\mathbf{Z}$. Moreover, if $Drin_A(k)$ is trivial,  one gets
 an inclusion $Spec~(A)\subset\mathbf{Q}/\mathbf{Z}$.
Equivalently,  each irreducible polynomial $a\in A$ defines an algebraic torus
$\mathscr{A}_{\alpha}$,  if  $Drin_A(k)$ is non-trivial  or a rational torus $\mathscr{A}_{\mathbf{Q}}$, 
 if  $Drin_A(k)$ is trivial.  We apply  Theorem \ref{thm1.1} to an explicit class field theory for the number
fields;  see Corollary \ref{cor1.2}. To formalize our results, let us introduce
the following notation.

By $C^*(k \langle\tau_p\rangle)$ one understands a semigroup $C^*$-algebra  [Li 2017] \cite{Li1}
of the  left cancelative  semigroup generated by $\tau_p$ and all  $a_i\in A$ satisfying the commutation relations 
$\tau_p a_i=a_i^p\tau_p$. The direct sum of $C^*$-algebras is a $C^*$-algebra denoted by $\oplus$.  
 Finally, let $Spec~(A)$ be the spectrum of $A$, i.e.  a collection 
of  all prime ideals of $A$. 
Our main results can be formulated as follows.  
\begin{theorem}\label{thm1.1}
The following is true:

\medskip
(i) $\bigoplus_p C^*(k \langle\tau_p\rangle)\cong \mathbb{A}(S_{1,1})$;

\smallskip
(ii)   the Drinfeld module  $Drin_A(k)$  is equivalent to an inclusion
 $Spec~(A) \subset\overline{\mathbf{Q}}/\mathbf{Z}$, 
such that  $Drin_A(k)$  is trivial if and only if $Spec~(A)\subset\mathbf{Q}/\mathbf{Z}$;  

\smallskip
(iii)  each  irreducible polynomial $a\in A$ 
defines an algebraic torus  $\mathscr{A}_{\alpha}$
containing a dense sub-algebra of the non-commutative polynomials  $\mathbf{K}\langle U,V\rangle$ 
over a number field $\mathbf{K}\subset \overline{\mathbf{Q}}$  (over $\mathbf{Q}$, resp.),  
if  $Drin_A(k)$ is non-trivial (trivial, resp.);

\smallskip
(iv) the Galois group of the field extension $\left( \mathbf{K} | \mathbf{Q}(\alpha)\right)$ is isomorphic to $\left(A/(a)\right)^*$.   
 \end{theorem}
Theorem \ref{thm1.1} implies an explicit class field theory for the number fields. 
Let  $\mathscr{A}_{\alpha}$ be an algebraic torus and $p(x)\in\mathbf{Z}[x]$ the minimal polynomial
of $\alpha\in \overline{\mathbf{Q}}$. 
Denote by $\varepsilon>1$ the Perron-Frobenius eigenvalue of 
a non-negative matrix $B\in GL_m(\mathbf{Z})$, such that $\det (B-xI)=p(x)$; see formula  (\ref{eq3.19plus}).  
Let $\mathbf{k}$ be a number field and  $\mathscr{H}(\mathbf{k})$ its Hilbert class field,
i.e. the maximal abelian  unramified  extension of  $\mathbf{k}$. 
\begin{corollary}\label{cor1.2}
\begin{equation}\label{eq1.6}
\mathscr{H}(\mathbf{k})\cong
\begin{cases} \mathbf{k}\left(e^{2\pi i\alpha+\log\log\varepsilon}\right), & if ~\mathbf{k}\subset\mathbf{C},\cr
               \mathbf{k}\left(\cos 2\pi\alpha \times\log\varepsilon\right), & if ~\mathbf{k}\subset\mathbf{R}.
\end{cases}               
\end{equation}
\end{corollary}

\medskip
\begin{remark}
The special case of formula (\ref{eq1.6})  for  the imaginary quadratic fields $\mathbf{k}$ was
established in \cite{Nik2}. 
\end{remark}
\begin{remark}
An explicit class field theory for the number fields based on the Bost-Connes crossed product 
$C^*$-algebras was studied  in [Yalkinoglu  2013] \cite{Yal1}. 
\end{remark}

\medskip
The paper is organized as follows.  A brief review of the preliminary facts is 
given in Section 2. Theorem \ref{thm1.1} and Corollary \ref{cor1.2}  are 
proved in Section 3. The case of imaginary quadratic field $\mathbf{k}$ 
is considered in Section 4.

\section{Preliminaries}
We briefly review the noncommutative tori, Mundici and cluster $C^*$-algebras, and Drinfeld modules.
We refer the reader to [Mundici 1988] \cite{Mun1}, \cite{Nik1},  \cite[Section 1.1]{N} and [Rosen 2002] \cite[Chapters 12 \& 13]{R} 
for a detailed exposition.

\subsection{Noncommutative tori}
\subsubsection{$C^*$-algebras}
The $C^*$-algebra is an algebra  $\mathscr{A}$ over $\mathbf{C}$ with a norm 
$a\mapsto ||a||$ and an involution $\{a\mapsto a^* ~|~ a\in \mathscr{A}\}$  such that $\mathscr{A}$ is
complete with  respect to the norm, and such that $||ab||\le ||a||~||b||$ and $||a^*a||=||a||^2$ for every  $a,b\in \mathscr{A}$.  
Each commutative $C^*$-algebra is  isomorphic
to the algebra $C_0(X)$ of continuous complex-valued
functions on some locally compact Hausdorff space $X$. 
Any other  algebra $\mathscr{A}$ can be thought of as  a noncommutative  
topological space. 

\subsubsection{K-theory of $C^*$-algebras}
By $M_{\infty}(\mathscr{A})$ 
one understands the algebraic direct limit of the $C^*$-algebras 
$M_n(\mathscr{A})$ under the embeddings $a\mapsto ~\mathbf{diag} (a,0)$. 
The direct limit $M_{\infty}(\mathscr{A})$  can be thought of as the $C^*$-algebra 
of infinite-dimensional matrices whose entries are all zero except for a finite number of the
non-zero entries taken from the $C^*$-algebra $\mathscr{A}$.
Two projections $p,q\in M_{\infty}(\mathscr{A})$ are equivalent, if there exists 
an element $v\in M_{\infty}(\mathscr{A})$,  such that $p=v^*v$ and $q=vv^*$. 
The equivalence class of projection $p$ is denoted by $[p]$.   
We write $V(\mathscr{A})$ to denote all equivalence classes of 
projections in the $C^*$-algebra $M_{\infty}(\mathscr{A})$, i.e.
$V(\mathscr{A}):=\{[p] ~:~ p=p^*=p^2\in M_{\infty}(\mathscr{A})\}$. 
The set $V(\mathscr{A})$ has the natural structure of an abelian 
semi-group with the addition operation defined by the formula 
$[p]+[q]:=\mathbf{diag}(p,q)=[p'\oplus q']$, where $p'\sim p, ~q'\sim q$ 
and $p'\perp q'$.  The identity of the semi-group $V(\mathscr{A})$ 
is given by $[0]$, where $0$ is the zero projection. 
By the $K_0$-group $K_0(\mathscr{A})$ of the unital $C^*$-algebra $\mathscr{A}$
one understands the Grothendieck group of the abelian semi-group
$V(\mathscr{A})$, i.e. a completion of $V(\mathscr{A})$ by the formal elements
$[p]-[q]$.  The image of $V(\mathscr{A})$ in  $K_0(\mathscr{A})$ 
is a positive cone $K_0^+(\mathscr{A})$ defining  the order structure $\le$  on the  
abelian group  $K_0(\mathscr{A})$. The pair   $\left(K_0(\mathscr{A}),  K_0^+(\mathscr{A})\right)$
is known as a dimension group of the $C^*$-algebra $\mathscr{A}$.

\subsubsection{Noncommutative tori}
The noncommutative torus is a universal $C^*$-algebra $\mathscr{A}_{\theta}$ 
generated by two invertible elements $u$ and $v$ satisfying the commutation relation:
\begin{equation}\label{eq2.1}
vu=e^{2\pi i \theta}uv
\end{equation}
for a real constant $\theta\in \mathbf{R}/\mathbf{Z}$. 
The dimension group $\left(K_0(\mathscr{A}_{\theta}),  K_0^+(\mathscr{A}_{\theta})\right)$
is order-isomorphic to $(\mathbf{Z}^2,\Lambda^+)$, where $\Lambda^+=\mathbf{Z}+\mathbf{Z}\theta>0$. 
In particular, the $\mathscr{A}_{\theta}$ is Morita equivalent to $\mathscr{A}_{\theta'}$
if and only if  $\theta'={a\theta +b\over c\theta+d}$ for a matrix 
$\left(\small\begin{matrix} a &b\cr c &d\end{matrix}\right) \in SL_2(\mathbf{Z})$. 

\subsubsection{AF-algebras}
An {\it AF-algebra}  (Approximately Finite-dimensional $C^*$-algebra) is defined to
be the  norm closure of an ascending sequence of   finite dimensional
$C^*$-algebras $M_n$,  where  $M_n$ is the $C^*$-algebra of the $n\times n$ matrices
with entries in $\mathbf{C}$. Here the index $n=(n_1,\dots,n_k)$ represents
the  semi-simple matrix algebra $M_n=M_{n_1}\oplus\dots\oplus M_{n_k}$.
The ascending sequence mentioned above  can be written as 
\begin{equation}\label{eq2.2}
M_1\buildrel\rm\varphi_1\over\longrightarrow M_2
   \buildrel\rm\varphi_2\over\longrightarrow\dots,
\end{equation}
where $M_i$ are the finite dimensional $C^*$-algebras and
$\varphi_i$ the homomorphisms between such algebras.  
If $\varphi_i=Const$, then the AF-algebra $\mathscr{A}$ is called 
{\it stationary}. 
The homomorphisms $\varphi_i$ can be arranged into  a graph as follows. 
Let  $M_i=M_{i_1}\oplus\dots\oplus M_{i_k}$ and 
$M_{i'}=M_{i_1'}\oplus\dots\oplus M_{i_k'}$ be 
the semi-simple $C^*$-algebras and $\varphi_i: M_i\to M_{i'}$ the  homomorphism. 
One has  two sets of vertices $V_{i_1},\dots, V_{i_k}$ and $V_{i_1'},\dots, V_{i_k'}$
joined by  $a_{rs}$ edges  whenever the summand $M_{i_r}$ contains $a_{rs}$
copies of the summand $M_{i_s'}$ under the embedding $\varphi_i$. 
As $i$ varies, one obtains an infinite graph called the   Bratteli diagram of the
AF-algebra.  The matrix $A=(a_{rs})$ is known as  a  partial multiplicity matrix;
an infinite sequence of $A_i$ defines a unique AF-algebra.
If   $\mathbb{A}$ is a stationary AF-algebra, then   $A_i=Const$
for all $i\ge 1$.  
The  dimension group $\left(K_0(\mathbb{A}),  K_0^+(\mathbb{A})\right)$  is a complete invariant of the
Morita equivalence class of the AF-algebra $\mathbb{A}$, see e.g. \cite[Theorem 3.5.2]{N}.

\subsection{Mundici algebra}

\subsubsection{Cluster $C^*$-algebras}
The cluster algebra  of rank $n$ 
is a subring  $\mathcal{A}(\mathbf{x}, B)$  of the field  of  rational functions in $n$ variables
depending  on  variables  $\mathbf{x}=(x_1,\dots, x_n)$
and a skew-symmetric matrix  $B=(b_{ij})\in M_n(\mathbf{Z})$.
The pair  $(\mathbf{x}, B)$ is called a  seed.
A new cluster $\mathbf{x}'=(x_1,\dots,x_k',\dots,  x_n)$ and a new
skew-symmetric matrix $B'=(b_{ij}')$ is obtained from 
$(\mathbf{x}, B)$ by the   exchange relations [Williams 2014]  \cite[Definition 2.22]{Wil1}:
\begin{eqnarray}\label{eq2.3}
x_kx_k'  &=& \prod_{i=1}^n  x_i^{\max(b_{ik}, 0)} + \prod_{i=1}^n  x_i^{\max(-b_{ik}, 0)},\cr 
b_{ij}' &=& 
\begin{cases}
-b_{ij}  & \mbox{if}   ~i=k  ~\mbox{or}  ~j=k\cr
b_{ij}+{|b_{ik}|b_{kj}+b_{ik}|b_{kj}|\over 2}  & \mbox{otherwise.}
\end{cases}
\end{eqnarray}
The seed $(\mathbf{x}', B')$ is said to be a  mutation of $(\mathbf{x}, B)$ in direction $k$.
where $1\le k\le n$.  The  algebra  $\mathcal{A}(\mathbf{x}, B)$ is  generated by the 
cluster  variables $\{x_i\}_{i=1}^{\infty}$
obtained from the initial seed $(\mathbf{x}, B)$ by the iteration of mutations  in all possible
directions $k$.   The  Laurent phenomenon
 says  that  $\mathcal{A}(\mathbf{x}, B)\subset \mathbf{Z}[\mathbf{x}^{\pm 1}]$,
where  $\mathbf{Z}[\mathbf{x}^{\pm 1}]$ is the ring of  the Laurent polynomials in  variables $\mathbf{x}=(x_1,\dots,x_n)$
 [Williams 2014]  \cite[Theorem 2.27]{Wil1}.
In particular, each  generator $x_i$  of  the algebra $\mathcal{A}(\mathbf{x}, B)$  can be 
written as a  Laurent polynomial in $n$ variables with the   integer coefficients.

 The cluster algebra  $\mathcal{A}(\mathbf{x}, B)$  has the structure of an additive abelian
semigroup consisting of the Laurent polynomials with positive coefficients. 
In other words,  the $\mathcal{A}(\mathbf{x}, B)$ is a dimension group, see Section 2.1.6 or  
\cite[Definition 3.5.2]{N}.
The cluster $C^*$-algebra  $\mathbb{A}(\mathbf{x}, B)$  is   an  AF-algebra,  
such that $K_0(\mathbb{A}(\mathbf{x}, B))\cong  \mathcal{A}(\mathbf{x}, B)$.

\subsubsection{Cluster $C^*$-algebra $\mathbb{A}(S_{g,n})$}
Denote by $S_{g,n}$  the Riemann surface   of genus $g\ge 0$  with  $n\ge 0$ cusps.
 Let   $\mathcal{A}(\mathbf{x},  S_{g,n})$ be the cluster algebra 
 coming from  a triangulation of the surface $S_{g,n}$   [Williams 2014]  \cite[Section 3.3]{Wil1}. 
 We shall denote by  $\mathbb{A}(S_{g,n})$  the corresponding cluster $C^*$-algebra. 

Let $T_{g,n}$ be the Teichm\"uller space of the surface $S_{g,n}$,
i.e. the set of all complex structures on $S_{g,n}$ endowed with the 
natural topology. The geodesic flow $T^t: T_{g,n}\to T_{g,n}$
is a one-parameter  group of matrices $\left(\small\begin{matrix} e^t &0\cr 0 &e^{-t}\end{matrix}\right)$
acting on the holomorphic quadratic differentials on the Riemann surface $S_{g,n}$. 
Such a flow gives rise to a one parameter group of automorphisms 
\begin{equation}\label{eq2.4}
\sigma_t: \mathbb{A}(S_{g,n})\to \mathbb{A}(S_{g,n})
\end{equation}
called the Tomita-Takesaki flow on the AF-algebra $\mathbb{A}(S_{g,n})$. 
Denote by $Prim~\mathbb{A}(S_{g,n})$ the space of all primitive ideals 
of $\mathbb{A}(S_{g,n})$ endowed with the Jacobson topology. 
Recall (\cite{Nik1}) that each primitive ideal has a parametrization by a vector 
$\Theta\in \mathbf{R}^{6g-7+2n}$ and we write it 
$I_{\Theta}\in Prim~\mathbb{A}(S_{g,n})$
\begin{theorem}\label{thm2.1}
{\bf (\cite{Nik1})}
There exists a homeomorphism
$h:  Prim~\mathbb{A}(S_{g,n})\times \mathbf{R}\to \{U\subseteq  T_{g,n} ~|~U~\hbox{{\sf is generic}}\}$
given by the formula $\sigma_t(I_{\Theta})\mapsto S_{g,n}$;  the set $U=T_{g,n}$ if and only if
$g=n=1$.   The $\sigma_t(I_{\Theta})$
is an ideal of  $\mathbb{A}(S_{g,n})$ for all $t\in \mathbf{R}$ and 
 the quotient  algebra  $\mathbb{A}(S_{g,n})/\sigma_t(I_{\Theta})$
is  a non-commutative coordinate ring  of  the Riemann surface  $S_{g,n}$.  
\end{theorem}
Let $\phi\in Mod ~(S_{g,n})$ be a pseudo-Anosov automorphism of $S_{g,n}$ 
with the dilatation $\lambda_{\phi}>1$. If the Riemann surfaces $S_{g,n}$ and $\phi (S_{g,n})$
lie on the axis of $\phi$, then Theorem \ref{thm2.1} gives rise  the Connes invariant \cite[Section 4.2]{Nik1}:
\begin{equation}\label{eq2.5}
T(\mathbb{A}(S_{g,n}))=\{\log\lambda_{\phi} ~|~ \phi\in Mod ~(S_{g,n})\}. 
\end{equation}

\subsubsection{Mundici algebra $\mathbb{A}(S_{1,1})$}
The case $g=n=1$ was first studied in [Mundici 1988] \cite{Mun1}. 
One gets a parametrization of  $Prim~\mathbb{A}(S_{1,1})$ by $\theta\in\mathbf{R}/\mathbf{Z}$
and the quotient algebra:
\begin{equation}\label{eq2.6}
\mathbb{A}(S_{1,1})/\sigma_t(I_{\theta})\cong \sigma_t(\mathbb{A}_{\theta}),
\end{equation}
where $\mathbb{A}_{\theta}$ is the Effros-Shen algebra, such that 
$\mathscr{A}_{\theta}\subset \mathbb{A}_{\theta}$ \cite[Example 3.5.2 and Theorem 3.5.3]{N}. 
Moreover, each Anosov automorphism $\phi\in Mod ~(S_{1,1})$ is represented
by a matrix $A\in SL_2(\mathbf{Z})$, so that the Connes invariant (\ref{eq2.5}) 
can be written as $\log\lambda_A$, where $\lambda_A$ is the Perron-Frobenius eigenvalue
of $A$.

\subsection{Drinfeld modules}
The explicit class field theory for the function fields is strikingly simpler
than for the number fields. The generators of the maximal abelian unramified
extensions (i.e. the Hilbert class fields) are constructed using 
the concept of the Drinfeld module. Roughly speaking, such a module 
is an analog of the exponential function and a generalization of the Carlitz module. 
Nothing similar  exists at the number fields side, where  the explicit 
generators of abelian extensions are known only for the field of rationals
(Kronecker-Weber theorem)
and imaginary quadratic number fields (complex multiplication). 
Below we give some details
on the Drinfeld modules.

Let $k$ be a field.  A polynomial $f\in k[x]$ is said to be additive
in the ring $k[x,y]$ if $f(x+y)=f(x)+f(y)$. If $char ~k=p$, then it is verified 
directly that the polynomial $\tau(x)=x^p$ is additive. Moreover, each 
additive polynomial has the form $a_0x+a_1x^p+\dots+a_rx^{p^r}$. 
The set of all additive polynomials is closed under addition and composition 
operations. The corresponding ring is isomorphic to a ring $k\langle\tau_p\rangle$
of the non-commutative polynomials given by  the commutation relation:
\begin{equation}\label{eq2.7}
\tau_p a=a^p\tau_p, \qquad \forall a\in k. 
\end{equation}

Let $A=\mathbf{F}_q[T]$ and $k=\mathbf{F}_q(T)$. 
By the  Drinfeld module one understands a homomorphism
\begin{equation}\label{eq2.8}
\rho: A\to k\langle\tau_p\rangle,
\end{equation}
such that for all $a\in A$ the constant term of $\rho_a$ is $a$ and 
$\rho_a\not\in k$ for at least one $a\in A$.  
Consider a torsion module $\Lambda_{\rho}[a]=\{\lambda\in\bar k ~|~\rho_a(\lambda)=0\}$.
The following result describes the simplest case of the explicit class field theory for the function
fields.
\begin{theorem}\label{thm2.2}
For each non-zero $a\in A$ the function field $k\left(\Lambda_{\rho}[a]\right)$ 
is an abelian extension of $k$ whose Galois group is isomorphic to a subgroup
of $\left(A/(a)\right)^*$. 
\end{theorem}

\section{Proof}
\subsection{Proof of theorem \ref{thm1.1}}
For the sake of clarity, let us outline the main ideas.  
We use a similarity (\ref{eq1.1}) and (\ref{eq1.4}) 
between the ring $k\langle\tau_p\rangle$  and the Mundici algebra $\mathbb{A}(S_{1,1})$;
the primitive ideals of $\mathbb{A}(S_{1,1})$ define the noncommutative tori $\mathscr{A}_{\theta}$,
see (\ref{eq2.6}).  Specifically, we construct a self-adjoint representation  of the ring 
$k\langle\tau_p\rangle$ by bounded linear operators on a Hilbert space $\mathcal{H}$,
and compare it with the $C^*$-algebra  $\mathbb{A}(S_{1,1})$. 
An elegant proof of isomorphism (i) of Theorem \ref{thm1.1} is given by 
the K-theory of algebra  $\mathbb{A}(S_{1,1})$ which is essentially an additive
subgroup of the Laurent polynomials $\mathbf{Z}[s^{\pm 1}, t^{\pm 1}]$. 
The remaining items (ii)-(iv) of Theorem \ref{thm1.1} follow from (i) 
and the standard properties of the Mundici algebra  $\mathbb{A}(S_{1,1})$.
Let us pass to a detailed argument.

\bigskip
(i) Let $\mathcal{A}(S_{1,1})\subset\mathbf{Z}[x_1^{\pm 1}, x_2^{\pm 1}, x_3^{\pm 1}]$
be the cluster algebra of rank $3$ associated to the ideal triangulation of a torus with one cusp \cite[Example 1]{Nik1}. 
Let $p$ be a prime number, and denote by  $\mathcal{A}_p(S_{1,1})$ a sub-algebra of  $\mathcal{A}(S_{1,1})$
consisting of the Laurent polynomials whose coefficients are divisible by $p$. It is easy to verify that   $\mathcal{A}_p(S_{1,1})$
is again a dimension group under the addition of the Laurent polynomials. We say that  $\mathbb{A}_p(S_{1,1})$ is a
congruence sub-algebra of level $p$ of the Mundici algebra  $\mathbb{A}(S_{1,1})$, i.e. 
$K_0(\mathbb{A}_p(S_{1,1}))\cong \mathcal{A}_p(S_{1,1})$. We shall split the proof of item (i) of Theorem \ref{thm1.1}
in two lemmas. 
\begin{lemma}\label{lm3.1}
 $C^*(k \langle\tau_p\rangle)\cong \mathbb{A}_p(S_{1,1})$
\end{lemma}
\begin{proof}
(i) First notice, that elements of the cluster algebra $\mathcal{A}(S_{1,1})\subset\mathbf{Z}[x_1^{\pm 1}, x_2^{\pm 1}, x_3^{\pm 1}]$
are the Laurent polynomials depending on the two variables $s=\frac{x_1}{x_3}$ and $t=\frac{x_2}{x_3}$. 
Indeed, the Ptolemy relations for the Penner coordinates $x_i=\lambda(\gamma_i)$ on the Teichm\"uller space $T_{1,1}$
are homogeneous, see  [Williams 2014]  \cite[Proposition 3.10]{Wil1}. We divide such relation by any non-zero $x_i$,
say, $x_i=x_3$.  Thus one gets an inclusion of the rings:
\begin{equation}\label{eq3.1}
\mathcal{A}_p(S_{1,1})\subset \mathcal{A}(S_{1,1})\subset\mathbf{Z}[s^{\pm 1}, t^{\pm 1}].
\end{equation}

\medskip
(ii) If we assume $t=Const$ in (\ref{eq3.1}), then each element of  $\mathcal{A}_p(S_{1,1})$
is given by a Laurent polynomial in one variable $\mathbf{Z}[s^{\pm 1}]$ whose coefficients 
are divisible by $p$. 
Moreover, the map $p\mathbf{Z}\mapsto \mathbf{Z}\mod p$ defines an inclusion 
of the rings: 
\begin{equation}\label{eq3.2}
\mathcal{A}_p(S_{1,1})\subset\mathbf{F}_p[s^{\pm 1}],
\end{equation}
where $\mathbf{F}_p[s^{\pm 1}]$ is the ring of Laurent polynomials 
over the finite field $\mathbf{F}_p$. 

\medskip
(iii) On the other hand, the formal Laurent series $\mathbf{F}_p\left(\left(\frac{1}{s}\right)\right)$
are known to define a completion $k_{\infty}$ of the function field $k$ at infinity, see 
e.g.   [Rosen 2002] \cite[p. 210]{R}.  Therefore (\ref{eq3.2}) gives rise to an isomorphism 
$\mathcal{A}_p(S_{1,1})\cong k_{\infty}$, where  $k_{\infty}$ is an analog of the real line 
$\mathbf{R}$ for the field $k$. 
Moreover, to get an analog $c_{\infty}$ of the complex plane $\mathbf{C}=\mathbf{R}+i\mathbf{R}$
for $k$,
one needs to drop the restriction $t=Const$ imposed on (\ref{eq3.1}) in item (ii). In other words,
one gets an isomorphism: 
\begin{equation}\label{eq3.3}
\mathcal{A}_p(S_{1,1})\cong c_{\infty}. 
\end{equation}

\medskip
(iv)
Recall that 
 $K_0(\mathbb{A}_p(S_{1,1})) \cong \mathcal{A}_p(S_{1,1})$
and the AF-algebra $\mathbb{A}_p(S_{1,1})$ is uniquely defined by its K-theory. Therefore to finish the proof
 of Lemma \ref{lm3.1}, one needs to calculate the group   $K_0(C^*(k \langle\tau_p\rangle))$.

\medskip
(v) 
It can be done as follows. Recall that the ring  $k\langle\tau_p\rangle$ is generated by the elements $a_i\in k$
and $\tau_p$ satisfying relations (\ref{eq1.2}) for all $a_i\in k$.  Let us denote by $C(a_1, a_2, \dots)$ a maximal 
abelian subalgebra (masa) of the $k \langle\tau_p\rangle$ generated by all elements $a_i\in k$. 
 The masa  $C(a_1, a_2, \dots)$ is an AF-algebra. Indeed, one can write
$C(a_1, a_2, \dots)=\lim_{n\to\infty} C(a_1,  \dots, a_n)$, where $C(a_1, \dots, a_n)$
is a finite-dimensional $C^*$-algebra generated by the elements $a_1, \dots, a_n\in k$. 
Thus $C(a_1, a_2, \dots)$ is an AF-algebra, see Section 2.1.4. 

\medskip
(vi)  On the other hand, each commutative $C^*$-algebra is isomorphic 
to $C(X)$, where $X$ is a Hausdorff topological space and $C(X)$ is the 
$C^*$-algebra of complex valued functions on $X$. Moreover, since 
$C(a_1, a_2, \dots)$ is an AF-algebra, the space $X$ must be a Cantor set. 
 Thus one gets the following isomorphisms between the abelian groups:
\begin{equation}\label{eq3.4}
K_0(C(a_1, a_2, \dots))\cong K_0(C(X))\cong \prod_{n\in\mathbf{N}} \mathbf{Z}_{(n)},
\end{equation}
where $X$ is the Cantor set and the last term is the cartesian product
of infinite number of copies of $\mathbf{Z}$.  
In particular,  it follows from (\ref{eq3.4}) that 
$K_0(C(a_1, a_2, \dots))$ is an uncountable abelian group.

\medskip
(vii)  Since  $C(a_1, a_2, \dots)\subset C^*(k \langle\tau_p\rangle)$, we conclude that
\begin{equation}\label{eq3.5}
K_0(C(a_1, a_2, \dots))\subset K_0(C^*(k \langle\tau_p\rangle)).
\end{equation}
In particular, the  $K_0(C^*(k \langle\tau_p\rangle))$ is an uncountable
abelian group.  Specifically, since $C(a_1, a_2, \dots)$ is generated by elements of a countable 
field $k$, it transpires that:   
\begin{equation}\label{eq3.6}
K_0(C^*(k \langle\tau_p\rangle))\cong c_{\infty},
\end{equation}
where $c_{\infty}$ is the algebraic completion at infinity of the function field $k$. 

\medskip
(viii) To finish the proof of Lemma \ref{lm3.1},  it remains to compare (\ref{eq3.3}) with  (\ref{eq3.6}).
Since  the $K_0$-groups define the underlying AF-algebras up to an isomorphism (Section 2.1.4),
we conclude that   $C^*(k \langle\tau_p\rangle)\cong \mathbb{A}_p(S_{1,1})$.
Lemma \ref{lm3.1} is proved.
\end{proof}

\begin{lemma}\label{lm3.2}
 $\bigoplus_p C^*(k \langle\tau_p\rangle)\cong \mathbb{A}(S_{1,1})$.
 \end{lemma}
\begin{proof}
(i) It follows from Lemma \ref{lm3.1} that
\begin{equation}\label{eq3.7}
\bigoplus_p ~C^*(k \langle\tau_p\rangle)\cong \bigoplus_p ~\mathbb{A}_p(S_{1,1}). 
\end{equation}
Thus to prove Lemma \ref{lm3.2},  we need to demonstrate that the Mundici algebra    $\mathbb{A}(S_{1,1})$
is isomorphic to the direct sum $\bigoplus_p ~\mathbb{A}_p(S_{1,1})$.

\medskip
(ii) To prove the isomorphism $\mathbb{A}(S_{1,1})\cong \bigoplus_p ~\mathbb{A}_p(S_{1,1})$,
let us calculate the K-theory of the direct sum  [Blackadar 1986] \cite[Section 5.2.3]{B}:
\begin{equation}\label{eq3.8}
K_0 \left(\bigoplus_p ~\mathbb{A}_p(S_{1,1}) \right)\cong \bigoplus_p K_0(\mathbb{A}_p(S_{1,1}))
\cong \bigoplus_p \mathcal{A}_p(S_{1,1}). 
\end{equation}

\medskip
(iii) It is easy to see, that $\mathcal{A}_p(S_{1,1})=p \mathcal{A}(S_{1,1})$ is the principal order-ideal of the
cluster algebra  $\mathcal{A}(S_{1,1})$,  because all coefficients of the Laurent series corresponding to the cluster algebra $\mathcal{A}_p(S_{1,1})$ 
are divisible by $p$. Thus the linear operator $\pi$  acting on  a vector space of the Laurent polynomials by the formula
$\mathcal{A}(S_{1,1})\mapsto p \mathcal{A}(S_{1,1})$  is a projection, i.e. $\pi^2=\pi$.  The latter follows from the inclusion of the 
linear subspaces $p^2 \mathcal{A}(S_{1,1})\subset  p \mathcal{A}(S_{1,1})$. 

\medskip
(iv) Therefore in formula (\ref{eq3.8})  the direct sum of  linear subspaces  $\mathcal{A}_p(S_{1,1})$  in formula (\ref{eq3.8}) 
gives us the linear space $\mathcal{A}(S_{1,1})$, i.e.
\begin{equation}\label{eq3.9}
\bigoplus_p \mathcal{A}_p(S_{1,1}) = \mathcal{A}(S_{1,1}). 
\end{equation}

\medskip
(v) It remains to recall, that  $\mathcal{A}(S_{1,1})=K_0( \mathbb{A}(S_{1,1}))$
and  $\mathcal{A}_p(S_{1,1})=K_0( \mathbb{A}_p(S_{1,1}))$ for all primes $p$. 
Since the $K_0$-groups define the AF-algebras up to an isomorphism, we conclude that 
 $\mathbb{A}(S_{1,1})\cong \bigoplus_p ~\mathbb{A}_p(S_{1,1})$. 
 This argument finishes the proof of Lemma \ref{lm3.2}. 
\end{proof}

\medskip
Item (i) of Theorem \ref{thm1.1} follows from Lemma \ref{lm3.2}.

\bigskip
(ii) We pass  to the proof of item (ii) of  Theorem \ref{thm1.1}
by splitting it in two lemmas. 
\begin{lemma}\label{lm3.3}
 $
 Spec~(A)\subset\mathbf{R}/\mathbf{Z}.$
\end{lemma}
\begin{proof}
(i) Let  $Drin_A(k)$  be the Drinfeld module defined by a homomorphism
\begin{equation}\label{eq3.10}
\rho: A\to k\langle\tau_p\rangle. 
\end{equation}
It follows from Lemma \ref{lm3.1} that  representation 
$\mathfrak{r}:  ~k\langle\tau_p\rangle \longrightarrow \mathscr{B}(\mathcal{H})$
composed with $\rho$ gives us a homomorphism
\begin{equation}\label{eq3.11}
\mathfrak{r}\circ\rho: ~A\to \mathbb{A}_p(S_{1,1}) \subset \mathbb{A}(S_{1,1}),
\end{equation}
where $\mathbb{A}_p(S_{1,1})$ is a congruence sub-algebra of level $p$
of the Mundici algebra $\mathbb{A}(S_{1,1})$. 

\medskip
(ii)  Consider the primitive spectra of rings in (\ref{eq3.11}):
\begin{equation}\label{eq3.12}
Prim~(A)\subset Prim   ~\mathbb{A}(S_{1,1})\cong \mathbf{R}/\mathbf{Z},
\end{equation}
where one gets the inclusion of spectra since $Prim$ is a functor, and the isomorphism 
$\mathbb{A}(S_{1,1})\cong \mathbf{R}/\mathbf{Z}$ was described in Section 2.2.3.

\medskip
(iii) Recall that $A$ is a commutative ring. For such rings  $Prim~(A)\cong Spec~(A)$,
where $Spec~(A)$ is the prime spectrum of $A$. Therefore one gets  from (\ref{eq3.12}) the
inclusion: 
\begin{equation}\label{eq3.13}
Spec~(A)\subset \mathbf{R}/\mathbf{Z}.
\end{equation}
Lemma \ref{lm3.3} is proved.  
\end{proof}

\begin{lemma}\label{lm3.4}
 $Spec~(A) \subset\overline{\mathbf{Q}}/\mathbf{Z}$. 
\end{lemma}
\begin{proof}
(i)  Let $k$ be a geometric extension of the field $\mathbf{F}_q(T)$ 
corresponding to the Riemann surface $S_{g,n}$ \cite[Chapter 7]{R}. 
Denote by $Mod~(S_{g,n})$ the mapping class group, i.e. the group
of automorphisms of the topological surface $S_{g,n}$ modulo a homotopy,
see also Section 2.2.2.  The group   $Mod~(S_{g,n})$ acts on  $Spec~(A) \subset \mathbf{R}/\mathbf{Z}$
as follows.

\medskip
(ii) Consider  the boundary $\partial\mathbb{H}\cong\mathbf{R}$  of the upper half-plane
$\mathbb{H}$, and let us extend $Spec~(A)$ to $\mathbf{R}$ 
 by glueing the copies of the unit interval    $\mathbf{R}/\mathbf{Z}$. 
Since $\mathbb{H}$ is the universal cover of $S_{g,n}$, the action of  
 $Mod~(S_{g,n})$ extends  to  $\partial\mathbb{H}$, so that  $Spec~(A)$
 is invariant under the action. Recall that if $\phi\in Mod~(S_{g,n})$
 is a pseudo-Anosov automorphism, then the foot points of the geodesic half-circle
 fixed by $\phi$ must be conjugate algebraic numbers $x,\bar x\in Spec~(A)$
 lying in the number field  $\mathbf{Q}(\lambda_{\phi})$, where 
 $\lambda_{\phi}>1$ is the dilatation of $\phi$. 
 The maximal degree of  $\mathbf{Q}(\lambda_{\phi})$
  is equal to $6g-6+2n$  [Thurston 1988] \cite[p. 427]{Thu1}. 
 Moreover, each $x\in  Spec~(A)$ can be obtained from a pseudo-Anosov 
 element $\phi\in Mod~(S_{g,n})$. Thus  $Spec~(A) \subset\overline{\mathbf{Q}}/\mathbf{Z}$.
 This argument finishes the proof of Lemma \ref{lm3.4}
\end{proof}

\begin{corollary}\label{cor3.5}
$Drin_A(k)$  is trivial if and only if $Spec~(A)\subset\mathbf{Q}/\mathbf{Z}$. 
\end{corollary}
\begin{proof}
(i) Let $Drin_A(k)$  be trivial Drinfeld module.  Since $\rho_a\in k$ for all $a\in A$,
we conclude that the geometric extension of the field $\mathbf{F}_q(T)$ is given by the Riemann
surface of  genus zero.  The group  $Mod~(S_{g,n})$ is trivial  and the dilatation $\lambda_{\phi}=1$. 
The degree of the extension $\mathbf{Q}(\lambda_{\phi})$ is one, and thus $Spec~(A)\subset\mathbf{Q}/\mathbf{Z}$. 

\medskip
(ii)  Conversely, if $Spec~(A)\subset\mathbf{Q}/\mathbf{Z}$,  then  $\lambda_{\phi}=1$ for all pseudo-Anosov elements 
$\phi\in Mod~(S_{g,n})$. The latter happens if and only if the surface $S_{g,n}$ has genus zero, i.e. $k\cong \mathbf{F}_q(T)$.
Therefore  $Drin_A(k)$ is trivial.  Corollary \ref{cor3.5} is proved. 
\end{proof}

\bigskip
(iii)  Let us  prove   item (iii) of  Theorem \ref{thm1.1}. 
Recall that  the polynomial $a\in  A$ is irreducible if and only if the principal  ideal $aA$ is 
a prime ideal in the ring $A$.   Since $Spec ~(A)$ consists of the prime  ideals, we simply write $a \in Spec ~(A)$. 
It follows from item (ii) of Theorem \ref{thm1.1} that there exists a  map $a\mapsto \mathscr{A}_{\alpha}$,
where $\mathscr{A}_{\alpha}$ is an algebraic noncommutative torus.

To define a dense sub-algebra $\mathbf{K}\langle U,V\rangle$ of the algebraic torus  $\mathscr{A}_{\alpha}$,
recall that there exists a covariant functor $F: \mathcal{E}(\mathbf{K})\to \mathscr{A}_{\alpha}$, where $\mathcal{E}(\mathbf{K})$
is a non-singular elliptic curve defined over a number field $\mathbf{K}\subset\overline{\mathbf{Q}}$
\cite[Section 1.3]{N}.  We set:
\begin{equation}\label{eq3.14}
\mathbf{K}\langle U,V\rangle\cong S(\alpha,\beta,\gamma), 
\end{equation}
where $\{S(\alpha,\beta,\gamma) ~|~\alpha,\beta,\gamma\in \mathbf{K}\}$ is the Sklyanin
algebra of $\mathcal{E}(\mathbf{K})$ \cite[p.11]{N}. 

If  $Drin_A(k)$ is trivial, then  $\mathscr{A}_{\mathbf{Q}}$ is a rational torus. 
It follows from \cite[Section 1.3]{N} that
 $\mathcal{E}(\mathbf{K})$ is a singular elliptic curve with  a rational parametrization. In particular, 
one gets $\mathbf{K}\cong\mathbf{Q}$ in this case. Item (iii) of Theorem \ref{thm1.1} is proved. 

\bigskip
(iv) Finally, let us  prove   item (iv) of  Theorem \ref{thm1.1}. 
Let $(a)\subset A$ be principal ideal generated by an irreducible polynomial $a\in A$. 
Consider an exact sequence of the ring homomorphisms:
\begin{equation}\label{eq3.15}
(a)\to A\to A/(a),
\end{equation}
where $A/(a)$ is a finite field with $q^{\deg a}:=h$ elements. 
The sequence (\ref{eq3.15}) can be viewed as
distinct ring inclusions $(a)\subset A$ classified by elements of the group $(A/(a))^*$. 
We let $(a)_1,\dots, (a)_h$ be prime ideals of $A$ corresponding
to such inclusions. 

In view of item (iii) of Theorem \ref{thm1.1},  one gets  the algebraic tori  
$\mathscr{A}_{\alpha_1}, \dots, \mathscr{A}_{\alpha_h}$
with the same dense sub-algebra $\mathbf{K}\langle U,V\rangle$. 
Likewise, we denote by $\beta_1,\dots,\beta_h\in\mathbf{K}$ the 
pairwise distinct generators of the field $\mathbf{K}$ defined by 
the above tori. Since the group  $(A/(a))^*$ acts transitively on $\beta_i$,
we conclude that 
\begin{equation}\label{eq3.16}
Gal ~\left( \mathbf{K} | \mathbf{Q}(\alpha)\right)\cong (A/(a))^*.
\end{equation}

Item (iv) of Theorem \ref{thm1.1} follows from (\ref{eq3.16}).

\bigskip
Theorem \ref{thm1.1} is proved.

\subsection{Proof of corollary \ref{cor1.2}}
For the sake of clarity, let us outline the main ideas.
Our proof is based on item (iv) of Theorem \ref{thm1.1}. 
We look for a generator $\beta$ of the field $\mathbf{K}$
in terms of invariants of the algebraic torus $\mathscr{A}_{\alpha}$.
The Sklyanin algebra  $\mathbf{K}\langle U,V\rangle$ provides 
a glimpse of how the elements of $\mathbf{K}$ should look like. 
Namely, since (\ref{eq1.1}) is the commutation relation in algebra
 $\mathbf{K}\langle U,V\rangle$,  the constant $e^{2\pi i\alpha}$ 
 must belong to the number field $\mathbf{K}$.  However, 
$e^{2\pi i\alpha}$ is a transcendental number unless $\alpha\in\mathbf{Q}$.
Thus one needs to consider all  images  $\{\varphi^t(\mathscr{A}_{\alpha}) ~|~t\in\mathbf{R}\}$
of  $\mathscr{A}_{\alpha}$ under the one-parameter group of  automorphisms $\varphi^t$.
The relation (\ref{eq1.1}) takes the form: 
\begin{equation}\label{eq3.17}
VU=te^{2\pi i \alpha}UV. 
\end{equation}
The Connes invariant $t=\log\varepsilon$  links the inner automorphism defined  by $\varphi^t$ 
 with the outer automorphisms of  $\mathscr{A}_{\alpha}$  coming from the pseudo-Anosov elements of 
the mapping class  group of surface $S_{g,n}$, see (\ref{eq2.5}).  (The value of the algebraic number $\varepsilon>1$
was specified in Section 1.) Thus one gets from (\ref{eq3.17})  a generator:
\begin{equation}\label{eq3.18}
\beta=(\log\varepsilon) e^{2\pi i \alpha}=e^{2\pi i \alpha+\log\log\varepsilon} \in\mathbf{K}.
\end{equation}
The real abelian extensions  are treated similarly and  use a real $C^*$-algebra
corresponding to $\mathscr{A}_{\alpha}$. Let us pass to a detailed argument.

\begin{lemma}\label{lm3.6}
$\mathscr{H}(\mathbf{k})\cong  \mathbf{k}\left(e^{2\pi i\alpha+\log\log\varepsilon}\right)$,  where  $\mathbf{k}\subset\mathbf{C}$. 
\end{lemma}
\begin{proof}
(i) Let $\mathbf{k}$ be a complex number field and  $\mathscr{A}_{\alpha}$
the corresponding
\footnote{There exists a  functor $F$ between the category of number fields $\mathbf{k}$
and such of the algebraic tori $\mathscr{A}_{\alpha}$.  For instance, $F(\mathbf{Q}(\sqrt{-D}))=\mathscr{A}_{\sqrt{D}}$, where $D>1$ is a
square-free  integer number;  see Section 4 the details.  A general explicit formula for $F$ is unknown to the author.}
algebraic torus. We let  the degree of $\alpha\in\overline{\mathbf{Q}}$ to be   $6g-6+2n$ for some
positive integers $g$ and $n$; see [Thurston 1988] \cite[p. 427]{Thu1}. 
 As explained in item (ii) of the proof of Lemma \ref{lm3.4},  the point $\alpha\in\partial\mathbb{H}$
is fixed by a pseudo-Anosov automorphism $\phi\in Mod~(S_{g,n})$. 

\medskip
(ii) Recall that the Tomita-Takesaki flow  $\sigma_t: \mathbb{A}(S_{g,n})\to \mathbb{A}(S_{g,n})$ gives rise 
to a one-paramter group of automorphisms $\varphi^t: \mathscr{A}_{\alpha}\to  \mathscr{A}_{\alpha}$
defined by the formula   $\varphi^t(\mathscr{A}_{\alpha})=\mathbb{A}(S_{g,n})/\sigma_t(I_{\Theta})$;
we refer the reader to Section 2.2.2 and Theorem \ref{thm2.1} for the notation and details. 

\medskip
(iii) To calculate the commutation relation (\ref{eq1.1}) for $\varphi^t(\mathscr{A}_{\alpha})$,
we must solve the equation:
\begin{equation}\label{eq3.19}
Tr~(\varphi^t(\mathscr{A}_{\alpha}))=Tr' ~(\mathscr{A}_{\alpha}), 
\end{equation}
where $Tr$ is the trace of $C^*$-algebra. It is easy to see, 
that  (\ref{eq3.19}) with $Tr'=t~Tr$ applied to (\ref{eq1.1}) gives  the commutation 
relation for  the $C^*$-algebra  $\varphi^t(\mathscr{A}_{\alpha})$
of the form $VU=t e^{2\pi i \alpha}UV$.
Indeed, equation (\ref{eq3.19}) is equivalent to $Tr~(VU)=t ~Tr~(e^{2\pi i\alpha}UV)=Tr~(te^{2\pi i\alpha}UV)$.
The latter is satisfied if and only if  $VU=t e^{2\pi i \alpha}UV$.

\medskip
(iv) On the other hand, the Riemann surfaces $S_{g,n}$ and $\phi(S_{g,n})$ always lie on the axis of
a pseudo-Anosov automorphism $\phi\in Mod~(S_{g,n}))$, see Section 2.2.2. The Connes invariant
$t=\log\lambda_{\phi}$, where $\lambda_{\phi}$ is the dilatation of $\phi$, see formula (\ref{eq2.5}).  
In other words, $\varphi^1(\mathscr{A}_{\alpha})\cong \varphi^t(\mathscr{A}_{\alpha})$ if and only if
$t=\log\lambda_{\phi}$. Moreover, it is easy to see that the dilatation $\lambda_{\phi}$ is equal to the algebraic
unit $\varepsilon>1$ expressed in terms of $\alpha$ as given in its definition,  see Section 1. 
Indeed, let $p(x)=x^m-a_{m-1}x^{m-1}-\dots-a_1x-a_0$ be the minimal polynomial of $\alpha\in\overline{\mathbf{Q}}$,
where $m=6g-6+2n$. It is well known that the matrix $B$ such that $\det (B-xI)=p(x)$ has the form:
\begin{equation}\label{eq3.19plus}
B=\left(
\begin{matrix}
0 & 0 &\dots &0& a_0\cr
1 & 0 &\dots &0& a_1\cr
\vdots & \vdots & && \vdots\cr
0 & 0  & \dots & 1 & a_{m-1}          
\end{matrix}
\right)
\end{equation}
On the other hand, the action of $\phi$ on the cohomology $H^1(S_{g,n}; \partial S_{g,n})$ 
is given by the matrix $B$, so that $\lambda_{\phi}$ is the Perron-Frobenius eigenvalue $\varepsilon>1$
of $B$ [Thurston 1988] \cite[p. 427]{Thu1}.   Therefore the 
commutation relation (\ref{eq1.1}) takes the form:
\begin{equation}\label{eq3.20}
VU=(\log\varepsilon) ~e^{2\pi i \alpha}UV=e^{2\pi i \alpha+\log\log\varepsilon}UV. 
\end{equation}

\medskip
(v) To finish the proof of Lemma \ref{lm3.6}, recall that $a\in A$ is an irreducible polynomial and, 
therefore, $(a)=pA$ is a prime ideal.  In particular,  the abelian group 
$\left(A/(a)\right)^*\cong Gal ~\left( \mathbf{K} | \mathbf{Q}(\alpha)\right)$
corresponds to the maximal unramified extension $\mathbf{K}$, see  item (iv) of Theorem \ref{thm1.1}.
But (\ref{eq3.20}) says that $\beta=  e^{2\pi i \alpha+\log\log\varepsilon}$ is a generator of $\mathbf{K}$ 
over $\mathbf{k}$. Thus $\mathscr{H}(\mathbf{k})\cong  \mathbf{k}\left(e^{2\pi i\alpha+\log\log\varepsilon}\right)$.
Lemma \ref{lm3.6} is proved. 
\end{proof}

\begin{lemma}\label{lm3.7}
$\mathscr{H}(\mathbf{k})\cong   \mathbf{k}\left(\cos 2\pi\alpha \times\log\varepsilon\right)$,  where  $\mathbf{k}\subset\mathbf{R}$. 
\end{lemma}
\begin{proof}
(i) Recall that a real $C^*$-algebra is a Banach $^*$-algebra $\mathscr{A}$ over
$\mathbf{R}$  isometrically $^*$-isomorphic to a norm-closed $^*$-algebra of bounded
operators on a real Hilbert space, see  [Rosenberg 2016] \cite{Ros1} for an excellent introduction. 
Given a real $C^*$-algebra $\mathscr{A}$, its complexification  $\mathscr{A}_{\mathbf{C}} = \mathscr{A} + i\mathscr{A}$  is a
complex $C^*$-algebra. Conversely, given a complex $C^*$-algebra  $\mathscr{A}_{\mathbf{C}}$, whether there 
exists a real $C^*$-algebra $\mathscr{A}$ such that $\mathscr{A}_{\mathbf{C}} = \mathscr{A} + i\mathscr{A}$ is unknown in general. 
However, the noncommutative torus admits the unique canonical decomposition:
\begin{equation}\label{eq3.21}
\mathscr{A}_{\theta}=\mathscr{A}_{\theta}^{Re}+i\mathscr{A}_{\theta}^{Re}. 
\end{equation}

\medskip
(ii)  Let  $\mathscr{A}_{\alpha}=\mathscr{A}_{\alpha}^{Re}+i\mathscr{A}_{\alpha}^{Re}$ be an algebraic torus. 
It is easy to verify, that the commutation relation (\ref{eq3.20}) for the real $C^*$-algebra $\mathscr{A}_{\alpha}^{Re}$ 
has the form: 
\begin{equation}\label{eq3.22}
VU=Re \left(e^{2\pi i \alpha+\log\log\varepsilon}\right)UV= \left(\cos 2\pi\alpha \times\log\varepsilon\right)UV. 
\end{equation}
We conclude from (\ref{eq3.22}) that $\mathscr{H}(\mathbf{k})=\mathbf{k}(\cos 2\pi\alpha \times\log\varepsilon)$,
where $\mathbf{k}\subset\mathbf{R}$.
This argument finishes the proof of Lemma \ref{lm3.7}. 
\end{proof}

\bigskip
Corollary \ref{cor1.2} follows from Lemmas \ref{lm3.6} and \ref{lm3.7}.


\section{Imaginary quadratic number fields}
Hilbert and Weber constructed generators of the maximal abelian unramified 
extension $\mathscr{H}(\mathbf{k})$ of the imaginary quadratic field  $\mathbf{k}=\mathbf{Q}(\sqrt{-D})$
using the theory of  complex multiplication. 
Namely,  let $\mathscr{E}_{\tau}$ be an elliptic curve isomorphic to the complex torus $\mathbf{C}/(\mathbf{Z}+\mathbf{Z}\tau)$, 
where $\tau=\sqrt{-D}$.  The lattice $\Lambda=\mathbf{Z}+\mathbf{Z}\tau$ has  multiplication 
by the ring of algebraic integers of the field $\mathbf{k}$; hence the name.  Recall that the $j$-invariant of   $\mathscr{E}_{\tau}$ can be defined  
by the convergent  series $j(\tau)=\frac{1} {q}+744+196884q+\dots$, where  $q=e^{2\pi i\tau}$. 
Roughly speaking, the main result of Hilbert and Weber says that:
\begin{equation}\label{eq4.1}
\mathscr{H}(\mathbf{k})=\mathbf{k}\left(j(\sqrt{-D})\right).
\end{equation}

\bigskip
Corollary \ref{cor1.2} gives a different set of generators for the field $\mathscr{H}(\mathbf{k})$. 
Namely, 
if $\mathbf{k}=\mathbf{Q}(\sqrt{-D})$ then  for almost all square-free
values of $D$ we know that  $\alpha=\sqrt{D}$ and $\varepsilon=\frac{1}{2} (a+b\sqrt{D})$, where $(a,b)$ 
is the smallest solution of the Pell equation $x^2-Dy^2=\pm 4$, i.e. $\varepsilon$ is the fundamental unit
of the ring $O_{\mathbf{k}}$ \cite[Theorem  6.3.1]{N}.  One gets the following formula for the maximal 
abelian unramified extension of the imaginary quadratic fields. 
\begin{example}\label{exm4.1}
$\mathscr{H}(\mathbf{k})\cong  \mathbf{k}\left(e^{2\pi i\sqrt{D}+\log\log\left( \frac{a+b\sqrt{D}}{2}\right)}\right)$. 
\end{example}
\begin{remark}
Example \ref{exm4.1} was first constructed in \cite{Nik2}. 
\end{remark}

\bibliographystyle{amsplain}


\end{document}